\documentclass{amsart}
\usepackage{amsmath}
\usepackage{amsfonts}
\usepackage{graphics}
\usepackage{epsfig}
\usepackage{amssymb}
\usepackage{amscd}
\usepackage[all]{xy}
\usepackage{latexsym}
\usepackage{graphicx}
\usepackage{multirow}
\usepackage{geometry}

\newtheorem{thm}{Theorem}[section]

\newtheorem{lem}[thm]{Lemma}

\newtheorem{quest}[thm]{Question}

\theoremstyle{definition}

\newtheorem{rem}[thm]{Remark}
\newtheorem*{ack}{Acknowledgement}

\newtheorem{case}{Case}

\numberwithin{equation}{section}
\numberwithin{figure}{section}

\def\trace{{\text{\rm{trace}}}}

\def\rchi{{\hbox{\raise1.5pt\hbox{$\chi$}}}}
\def\Aut{{\text{\rm{Aut}}}}

\def\const{{\text{\rm{const}}}}

\def\a{\alpha}
\def\b{\beta}

\def\gam{\gamma}
\def\Gam{\Gamma}

\newcommand{\Mbar}{{\overline{\mathcal{M}}}}

\newcommand{\bP}{{\mathbb{P}}}
\newcommand{\bC}{{\mathbb{C}}}

\newcommand{\bR}{{\mathbb{R}}}

\newcommand{\bZ}{{\mathbb{Z}}}

\newcommand{\cM}{{\mathcal{M}}}

\newcommand{\cD}{{\mathcal{D}}}

\newcommand{\cH}{{\mathcal{H}}}

\newcommand{\la}{{\langle}}
\newcommand{\ra}{{\rangle}}
\newcommand{\half}{{\frac{1}{2}}}

\begin{document}
\large
\setcounter{section}{0}

\title[Laplace transform and Eynard-Orantin recursion]
{The Laplace transform, mirror symmetry, and the topological recursion of 
 Eynard-Orantin}

\author[M.\ Mulase]{Motohico Mulase}
\address{
Department of Mathematics\\
University of California\\
Davis, CA 95616--8633, U.S.A.}
\email{mulase@math.ucdavis.edu}

\begin{abstract}
This paper is based on the author's talk
at the 2012
Workshop on Geometric Methods in Physics held
 in Bia\l owie\.za, Poland.
The aim of the talk is to 
introduce the audience to the
Eynard-Orantin topological recursion. The 
 formalism
is originated 
in random matrix theory.  It has been predicted,
and in some cases it has been
proven, that the theory provides an effective 
 mechanism to calculate
certain quantum invariants and a solution
to enumerative geometry problems, such as 
open Gromov-Witten invariants of 
 toric Calabi-Yau threefolds, 
single and double  
Hurwitz numbers,  the number of
lattice points on the moduli space of 
smooth  algebraic curves,
and quantum knot invariants.
In this paper we use the Laplace
transform of generalized Catalan
numbers of an arbitrary genus as an example, and
present the Eynard-Orantin recursion. We examine
 various aspects of the theory, such
as its relations to mirror symmetry,
Gromov-Witten invariants,
integrable hierarchies such as the KP equations,
and the Schr\"odinger equations. 
\end{abstract}

\subjclass[2000]{Primary: 14H15, 14N35, 05C30, 11P21;
Secondary: 81T30}

\maketitle

\allowdisplaybreaks

\tableofcontents

\section{Introduction}
\label{sect:intro}

The purpose of this paper is to give an 
introduction to the Eynard-Orantin
topological recursion 
\cite{EO1}, by going through a simple 
mathematical example. 
Our example is constructed from
the Catalan numbers, 
their higher-genus analogues, and 
the mirror symmetry of these
numbers.

There have been exciting new
developments around the Eynard-Orantin
 theory in the last
few years that involve
various quantum topological invariants, such
as single and double 
Hurwitz numbers, 
open Gromov-Witten invariants, and 
 quantum knot polynomials. 
A big picture is being proposed, from which, 
for example,
we can understand
  the relation between the A-polynomial 
 \cite{CCGLS} of a knot 
and  its 
colored Jones polynomials as the same as the
\textbf{mirror symmetry} in string theory.

From the rigorous 
mathematical point of view, the  predictions 
on this subject coming from physics  
are conjectural.
In mathematics we need  a simple 
 example, for which
we can prove all the predicted
properties, and from which we
can see what is
 going on in a more general 
 context. The aim of this paper
is to present such an example of the 
Eynard-Orantin theory.

The formalism of our interest is originated
 in the large $N$ asymptotic analysis  of 
the correlation functions of resolvents of 
a random matrix of size $N\times N$
\cite{AMM, E2004}. The motivation of
Eynard and Orantin 
\cite{EO1} is to find applications
of the computational mechanism
beyond random matrix theory.
Their formula takes the shape of an integral
recursion equation on a given Riemann 
surface $\Sigma$
called the \textbf{spectral curve}
of the theory. 
At that time 
already  Mari\~no was developing the idea of 
 \textbf{remodeled B-model} of 
topological string theory on a Riemann
surface $\Sigma$ in \cite{M2}. He 
noticed
the geometric significance of \cite{EO1}, and
formulated a precise theory of remodeling
B-model
with Bouchard, Klemm,  and Pasquetti in
\cite{BKMP}. 
This work immediately attracted the attention of the 
mathematics community. 
The currently accepted picture
 is that the remodeled
B-model defines symmetric differential
forms  on $\Sigma$ via the Eynard-Orantin
recursion, and that these differentials forms are the 
\textbf{Laplace transform} of the quantum 
topological invariants that appear on the 
A-model side of the story. In this context 
\emph{the Laplace transform plays the role of the
mirror symmetry}.

This picture tells us that once we identify the
spectral curve $\Sigma$, we can calculate the 
quantum topological invariants in terms of
complex analysis on $\Sigma$.
The effectiveness of this mechanism has been 
mathematically proven for single
Hurwitz numbers \cite{EMS, MZ}, 
orbifold (or double)
Hurwitz numbers \cite{BHLM}, 
enumeration of the
lattice points of $\cM_{g,n}$
\cite{CMS, N1, N2}, the Poincar\'e polynomials
of  $\cM_{g,n}$ \cite{MP2012}, 
  the Weil-Petersson volume and its
  higher analogues of  $\Mbar_{g,n}$
  \cite{EO2, LX, Mir1, Mir2, MS},
and the higher-genus Catalan numbers
\cite{DMSS}. 
A spectacular conjecture of \cite{BKMP} states
that the Laplace transform of
the  open Gromov-Witten invariants of 
an arbitrary toric Calabi-Yau threefold 
satisfies the Eynard-Orantin topological 
recursion. A significant progress toward this
conjecture has been made in \cite{EO3}.

Furthermore, an unexpected  application of the 
Eynard-Orantin
theory has been proposed in knot theory
\cite{AV, BE2, BEM, DFM, FGS1, GS}. A
 key ingredient there is the \textbf{quantum curve}
 that characterizes
 quantum knot invariants.

The word \emph{quantum} means many different 
things in modern mathematics.
For example, a quantum curve is a 
holonomic system of linear
differential equations whose Lagrangian is an 
algebraic curve embedded
in the cotangent bundle of a base curve. 
Quantum knot invariants, on
the other hand, are invariants of knots defined by
representation theory of quantum algebras, and 
quantum algebras are
deformations of usual algebras. In such a diverse 
usage, the only
common feature is the aspect of non-commutative 
deformations.
Therefore, when two completely different quantum 
objects turn out
to be the same, we expect a deep mathematical theory
  behind
the scene. 
 In this vein, within the last two years, mathematicians
and physicists have discovered a new, miraculous 
 mathematical procedure,
although still conjectural, 
 that directly relates quantum curves
and quantum knot invariants.

    The notion of quantum curves appeared 
    in Aganagic, Dijkgraaf, Klemm,
    Mari\~no, and Vafa \cite{ADKMV}, and 
    later in Dijkgraaf, Hollands, Su\l kowski, and Vafa
     \cite{DHS, DHSV}. 
     When the A-model we start with has a
     vanishing obstruction class in algebraic
     K-theory, then it is expected that 
 a quantum curve exists, and it
  is a differential operator. 
 Let us call it $P$.
A quantum knot invariant is a function. 
 Call it $Z$. Then the
conjectural relation is simply
  the \textbf{Schr\"odinger equation} $PZ = 0$.
For this equation to make sense, in addition to the
very existence of $P$, 
   we need to identify the variables
appearing in $P$ and $Z$. The key observation 
   is that both $P$ and $Z$ are
defined on the same Riemann surface, 
   and that it is exactly the
spectral curve of the Eynard-Orantin topological 
   recursion, being
realized as a Lagrangian immersion.
   Moreover, the total symbol of the operator 
   $P$ defines the Lagrangian immersion.

What is the significance
of this Schr\"odinger equation $PZ = 0$, then?
      Recently Gukov and Su\l kowski \cite{GS}, 
based on 
      \cite{DFM}, 
      provided the crucial
insight that when the underlying 
spectral curve is
defined by the 
      A-polynomial of a knot,
      the algebraic K-theory obstruction vanishes,
      and
 the equation $PZ = 0$
       becomes the same as the AJ-conjecture of
Garoufalidis \cite{Gar}. This means that
the Eynard-Orantin
       theory
       conjecturally computes  colored Jones
       polynomial as
       the partition function $Z$ of 
the theory, starting from a given A-polynomial.

In what follows, we present a simple example
of the story. Although our example
 is not related to 
knot theory, it exhibits all key ingredients of the 
theory, such as the Schr\"odinger equation, 
relations to quantum topological invariants, 
the Eynard-Orantin recursion, 
the KP equations, and mirror symmetry.

At the Bia\l owie\.za Workshop in 
summer 2012, Professor
L.~D.~Faddeev gave a beautiful talk on the
quantum dilogarithm, Bloch groups, and
algebraic K-theory \cite{F}. Our example of this
paper does not illustrate the 
fundamental connection to 
these  important subjects, 
because our spectral curve (\ref{eq:x=z+1/z}) 
has genus $0$, and the K-theoretic obstruction to 
quantization, similar to the idea of
$K_2$-Lagrangian of Kontsevich,
 vanishes. Further developments
are expected
in this direction.

\section{Mirror dual of the Catalan numbers and
their higher genus extensions}
\label{sect:mirror}

The \textbf{Catalan numbers} appear in many 
different
places of mathematics and physics, often quite
unexpectedly. The \emph{Wikipedia} lists
some of the mathematical interpretations. 
The appearance in string theory \cite{OSV} is
surprising.
Here let us use the following definition:
\begin{equation}
\label{eq:parenthesis}
C_m = {\text{the number of ways to 
 place $2m$ pairs of parentheses in a legal manner}}.
\end{equation}
A \emph{legal} manner means the usual way we 
stack them together. 
If we have one pair, then $C_1=1$,
because  $(\;)$ is legal, while  
$) ($ is not.
For $m=2$, we have  $((\;))$ and $(\;)(\;)$, hence
$C_2=2$. Similarly,
$C_3 = 5$ because there are  five legal
combinations:
$$
(((\;))), ((\;))(\;),((\;)(\;)),(\;)((\;)),(\;)(\;)(\;).
$$
This way of exhaustive listing 
becomes harder and harder
as $m$ grows. We need a better mechanism to
find the value, and also a general 
\emph{closed} formula,
if at all possible. Indeed, we have the \emph{Catalan
recursion equation}
\begin{equation}
\label{eq:Catalan recursion}
C_m = \sum_{a+b=m-1} C_aC_b,
\end{equation}
and  a closed formula
\begin{equation}
\label{eq:Catalan formula}
C_m = \frac{1}{m+1}\binom{2m}{m}.
\end{equation}
Although our definition (\ref{eq:parenthesis})
does not make sense for $m=0$, 
the closed formula  (\ref{eq:Catalan formula})
tells us that $C_0=1$, and the recursion 
(\ref{eq:Catalan recursion}) works only if we 
define $C_0=1$.
We will give a proof of these formulas 
later.

Being a ubiquitous object, the Catalan numbers
have many different generalizations. What we
are interested here is not those kind of 
generalized Catalan numbers. We want to 
define \emph{higher-genus} Catalan numbers.
They are necessary if we ask the following question:

\begin{quest}
What is the mirror symmetric dual object of the
Catalan numbers?
\end{quest}

The mirror symmetry was  conceived
in modern theoretical physics as a duality between 
two  different
Calabi-Yau spaces of three complex dimensions. 
According to this idea, the universe consists of
the visible
$3$-dimensional spatial component, $1$-dimensional
time component, and an invisible $6$-dimensional
component. 
The invisible 
component of the universe is
considered as a complex $3$-dimensional 
Calabi-Yau space, and  the quantum nature of
the universe, manifested in quantum interactions
of elementary particles
and black holes, is believed to be
hidden in the geometric structure of 
this invisible manifold. The surprising discovery is 
that the same physical properties can be
obtained from two different settings: 
a Calabi-Yau space $X$ with its K\"ahler structure,
or another Calabi-Yau space $Y$ with its
complex structure. The duality between these 
two sets of data is the mirror symmetry.

The phrase,
``having the same quantum nature of the universe,''
does  not give a mathematical definition. The idea of
Kontsevich  
\cite{K1994}, the \textbf{Homological Mirror 
Symmetry}, is to   define the mirror 
symmetry as the equivalence of 
\emph{derived} categories. Since 
categories do not necessarily require
underlying spaces,  we can 
talk about mirror symmetries among more
general objects. 
For instance, we can 
ask the above question.

What I'd like to 
explain in this paper is that the answer to 
the question is a  
simple \emph{function}
\begin{equation}
\label{eq:x=z+1/z}
x = z + \frac{1}{z}.
\end{equation}
It is quite radical: the mirror symmetry holds
between the Catalan numbers and a function
like (\ref{eq:x=z+1/z})!

If we naively understand the homological mirror
symmetry as the derived equivalence
between symplectic geometry (the A-model side)
 and holomorphic complex geometry 
 (the B-model side), then
 it is easy to guess that (\ref{eq:x=z+1/z}) 
 should define a  B-model. According to 
Ballard \cite{B}, the mirror symmetric 
partner to this function is the projective
line $\bP^1$, together with its standard
K\"ahler structure. The higher-genus 
Catalan numbers we are going to define 
below are associated with the 
K\"ahler geometry of $\bP^1$. 
Their mirror symmetric partners are the
symmetric differential forms that
the Eynard-Orantin theory defines on the
Riemann surface of the function
$x=z+\frac{1}{z}$.

It is more convenient to give a different
definition of 
the Catalan numbers that makes the 
higher-genus 
extension more straightforward. Consider
a graph $\Gam$ drawn on a sphere 
$S^2$ that has only 
one vertex. Since every edge coming out
from this vertex has to come back, the vertex
has an even degree, say $2m$. This means
$2m$ \textbf{half-edges} are incident to the
unique vertex.
Let us place an outgoing arrow to one of 
the half-edges near at the  vertex
(see Figure~\ref{fig:graph01}).
Since $\Gam$ is drawn on $S^2$,
the large loop of the left  of 
Figure~\ref{fig:graph01} can be placed
as in the right graph. These are the same
graph on the sphere.

\begin{figure}[htb]
\centerline{
\epsfig{file=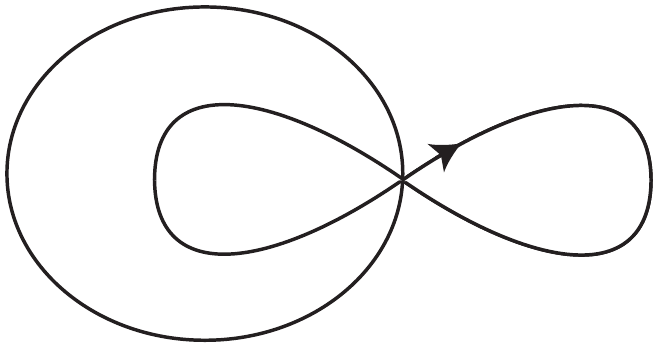, width=1.5in}
\qquad\qquad
\epsfig{file=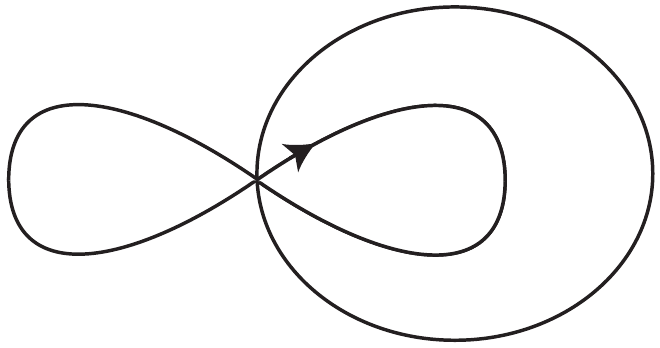, width=1.5in}}
\caption{Two ways
of representing the same
 arrowed graph on $S^2$ with one 
vertex. This graph corresponds to $(((\;)))$.
}
\label{fig:graph01}
\end{figure}

\begin{lem}
The number of arrowed graphs on 
$S^2$ with one vertex
of degree $2m$ is equal to the Catalan number
$C_{m}$.
\end{lem}

\begin{proof}
We assign to each edge forming a loop a
pair of parentheses. Their placement
is nested according to the graph. The
starting parenthesis `$($' corresponds to 
the unique arrowed half-edge. We then
examine all half-edges by the counter clock-wise
order. When a new loop is started, we open
a parenthesis `$($'. When it is closed to form
a loop, we complete a pair of parentheses 
by placing a `$)$'. In this way we have a
bijective correspondence between graphs on
$S^2$ with one vertex of degree $2m$ and the nested 
pairs of $2m$ parentheses. 
\end{proof}

Now a higher-genus generalization is easy.
A \textbf{cellular graph} of type $(g,n)$
is the one-skeleton of a cell-decomposition of 
a  connected, closed, oriented surface
of genus $g$ with $n$ 
$0$-cells labeled by the index set 
$[n]=\{1,2,\dots,n\}$. Two cellular graphs
are identified if an orientation-preserving 
homeomorphism  of a surface into another
surface maps one cellular graph to another,
honoring the labeling of each vertex.
Let $D_{g,n}(\mu_1,\dots, \mu_n)$ denote
the number of connected cellular graphs $\Gam$ of
type $(g,n)$ with $n$ labeled vertices
of degrees $(\mu_1,\dots,\mu_n)$, 
counted with the weight 
$1/|\Aut(\Gam)|$. It is generally a rational number.
The orientation of the surface induces a 
cyclic order of incident half-edges at each 
vertex of a cellular graph $\Gam$. Since 
$\Aut(\Gam)$ fixes each vertex, 
it is a subgroup of the Abelian group
$\prod_{i=1} ^n \bZ\big/\mu_i \bZ$ that rotates
each vertex and the incident half-edges.
Therefore, 
\begin{equation}
\label{eq:Catalan gn}
C_{g,n}(\mu_1,\dots,\mu_n)
= \mu_1\cdots\mu_n D_{g,n}(\mu_1,\dots,\mu_n)
\end{equation}
is always an integer. 
The cellular graphs counted by
(\ref{eq:Catalan gn}) are connected 
 graphs of genus $g$ with $n$ vertices of degrees
$(\mu_1,\dots,\mu_n)$, and at the $j$-th vertex 
for every $j=1,\dots,n$, an 
arrow is placed on one of the incident
$\mu_j$ half-edges (see Figure~\ref{fig:cellulargraph}).
The placement of $n$ arrows corresponds to the
factors $\mu_1\cdots\mu_n$ on the right-hand side.
We call this integer the 
\textbf{Catalan number} of type $(g,n)$.
The reason for this naming comes from the
fact that $C_{0,1}(2m) = C_m$, and the
following theorem.

\begin{figure}[htb]
\centerline{
\epsfig{file=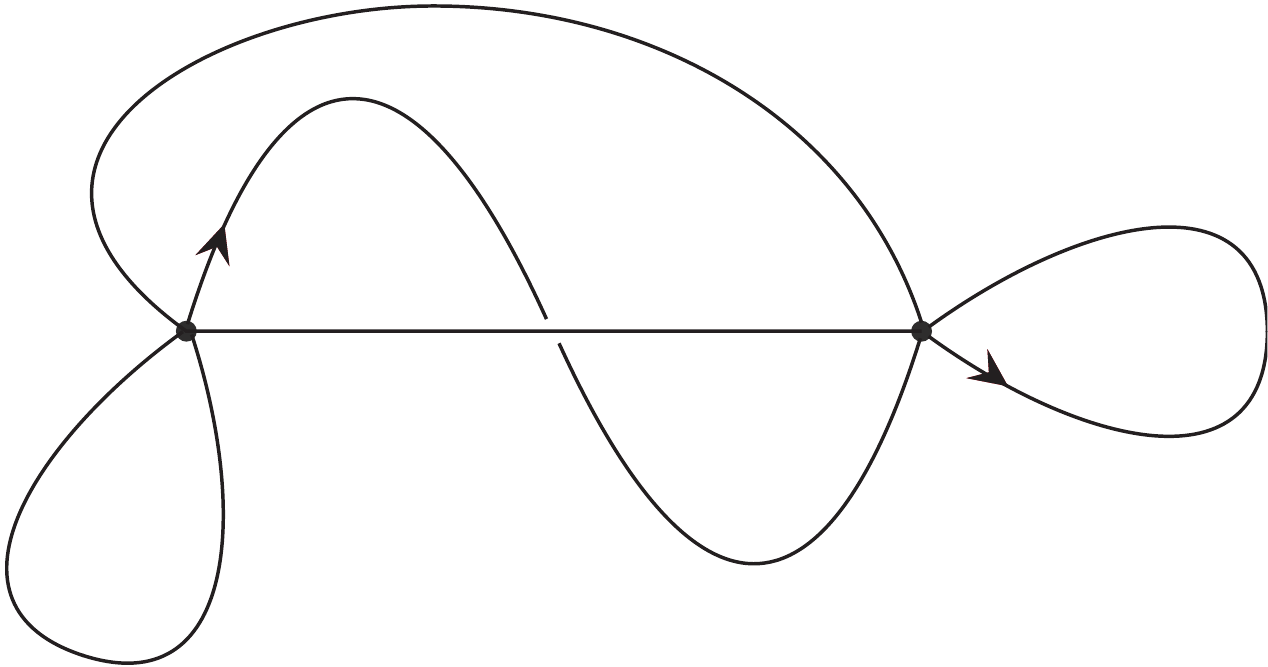, width=2in}}
\caption{A cellular graph of type $(1,2)$.}
\label{fig:cellulargraph}
\end{figure}

\begin{thm}
The generalized Catalan numbers of 
{\rm(\ref{eq:Catalan gn})}
satisfy the following equation.
\begin{multline}
\label{eq:Catalan gn recursion}
C_{g,n}(\mu_1,\dots,\mu_n) 
=\sum_{j=2}^n \mu_j 
C_{g,n-1}(\mu_1+\mu_j-2,\mu_2,\dots,
\widehat{\mu_j},\dots,\mu_n)
\\
+
\sum_{\a+\b = \mu_1-2}
\left[
C_{g-1,n+1}(\a,\b,\mu_2,\cdots,\mu_n)+
\sum_{\substack{g_1+g_2=g\\
I\sqcup J=\{2,\dots,n\}}}
C_{g_1,|I|+1}(\a,\mu_I)C_{g_2,|J|+1}(\b,\mu_J)
\right],
\end{multline}
where $\mu_I=(\mu_i)_{i\in I}$ for 
an index set $I\subset[n]$,
$|I|$ denotes the cardinality of $I$, and
the third sum in the formula is for 
all  partitions of $g$ and 
set partitions of $\{2,\dots,n\}$. 
\end{thm}

\begin{proof}
Consider an arrowed cellular graph $\Gam$
counted
by the left-hand side of 
(\ref{eq:Catalan gn recursion}), and 
let
$\{p_1,\dots,p_n\}$ denote the set of 
labeled vertices of $\Gam$.
We look at the half-edge incident to $p_1$ that
carries an arrow. 

\begin{case}
The arrowed half-edge  extends to an edge $E$ that
connects $p_1$ and $p_j$ for some $j>1$.
\end{case}

We shrink the edge $E$ and join the two
vertices $p_1$ and $p_j$ together. By this process
we create a new vertex of degree $\mu_1+\mu_j-2$.
To make the counting bijective, we need to be able
to go back from the shrunken graph to the original,
provided that we know $\mu_1$ and $\mu_j$.
Thus we place an arrow to the half-edge 
next to $E$ around $p_1$ with respect to the
counter-clockwise cyclic order that comes from 
the orientation of the surface. In this process
we have $\mu_j$ different arrowed graphs 
that produce the same result, because we must
remove the arrow placed around the vertex $p_j$
in the original graph. This gives the 
right-hand side of the first line of 
(\ref{eq:Catalan gn recursion}). See 
Figure~\ref{fig:case1}.

\begin{figure}[htb]
\centerline{
\epsfig{file=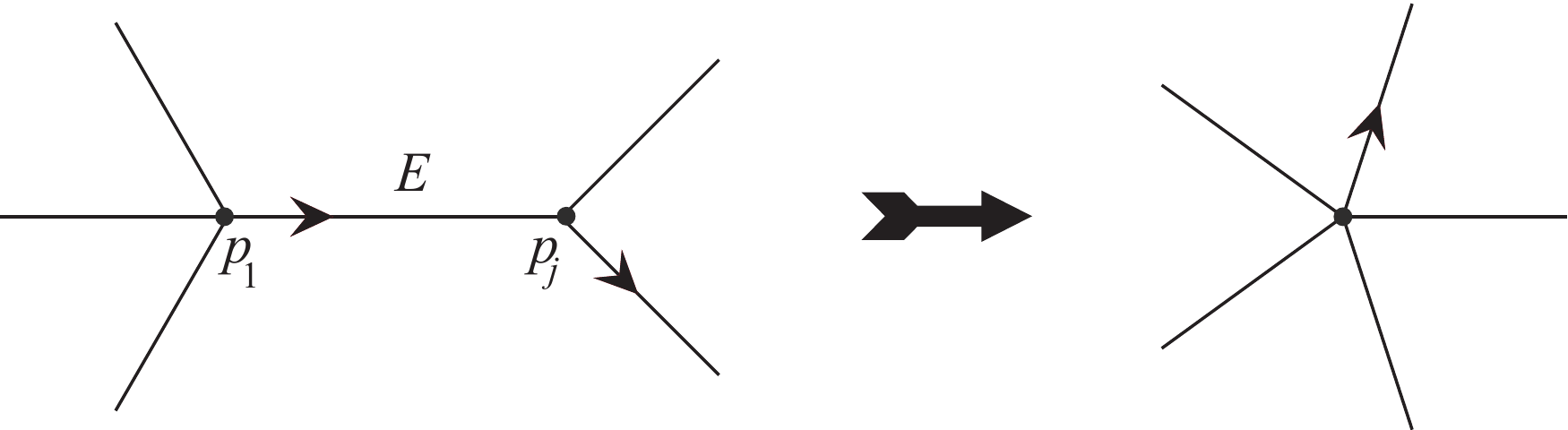, width=2.5in}}
\caption{The process of shrinking the arrowed edge
$E$ that connects vertices $p_1$ and $p_j$, $j>1$.}
\label{fig:case1}
\end{figure}

\begin{case}
The arrowed half-edge at $p_1$ is actually a loop
$E$ that goes out and comes back to $p_1$.
\end{case}

The process we apply is again shrinking the loop $E$.
The loop $E$ separates all other half-edges into 
two groups, one consisting of $\a$ of them placed on
one side of the loop, and the other consisting of 
$\b$ half-edges placed on the other side. It can 
happen that $\a=0$ or $\b=0$. 
Shrinking a loop on a surface causes pinching. 
Instead of creating a pinched (i.e., singular) surface, 
we separate the double point into two new vertices
of degrees $\a$ and $\b$. 
Here again we need to remember the position of the 
loop $E$. Thus we place an arrow to the half-edge
next to the loop in each group.
See Figure~\ref{fig:case2}.

\begin{figure}[htb]
\centerline{
\epsfig{file=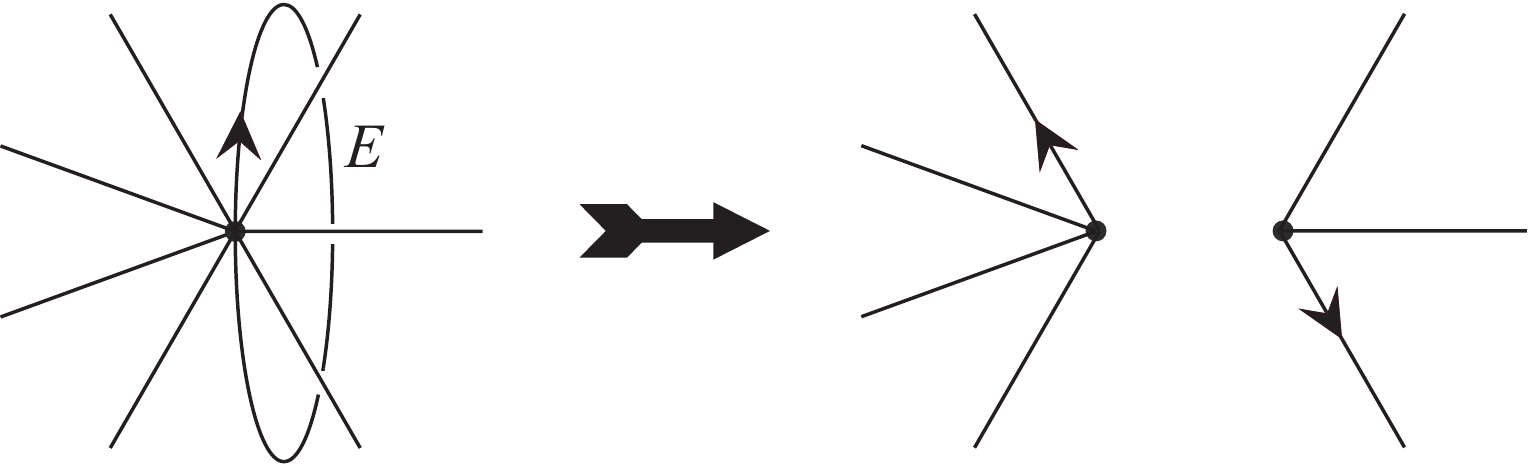, width=2.5in}}
\caption{The process of shrinking the arrowed loop
$E$ that is attached to  $p_1$.}
\label{fig:case2}
\end{figure}

After the pinching and separating the double 
point, the original surface of genus $g$ with 
$n$ vertices $\{p_1,\dots,p_n\}$ may change its
topology. It may have genus $g-1$, or it splits into
two pieces of genus $g_1$ and $g_2$. 
The second line of (\ref{eq:Catalan gn recursion})
records all such possibilities.
 This completes the proof.
\end{proof}

\begin{rem}
For $(g,n) = (0,1)$, the above formula reduces to 
\begin{equation}
\label{eq:Cm}
C_{0,1}(\mu_1) =
\sum_{\a+\b=\mu_1-2} C_{0,1}(\a) C_{0,1}(\b),
\end{equation}
which proves (\ref{eq:Catalan recursion}) since
$C_{0,1}(2m) = C_m$.
\end{rem}

Note that we \emph{define} 
$C_{0,1}(0)=1$. 
Only for the $(g,n)=(0,1)$ case
 this irregularity   of non-zero
value happens  for $\mu_1=0$. This is
 because a degree $0$ single vertex is
\emph{connected}, and gives a cell-decomposition
of $S^2$. We can imagine that a single vertex on
$S^2$ has
an infinite cyclic group as its automorphism, 
so that $C_{0,1}(0)=1$ is consistent. 
In all other cases,
if one of the vertices has degree $0$, then the 
Catalan number $C_{g,n}$ 
is simply $0$ because of the
definition
(\ref{eq:Catalan gn}).

Following Kodama-Pierce \cite{KP}, 
we introduce the generating function of the 
Catalan numbers by
\begin{equation}
\label{eq:Catalan z}
z=z(x) = \sum_{m=0} ^\infty 
C_m\frac{1}{x^{2m+1}}.
\end{equation}
Then by the quadratic recursion (\ref{eq:Cm}),
we find that the inverse function of $z(x)$
that vanishes at $x=\infty$ is given by
$$
x=z+\frac{1}{z},
$$
which is exactly
(\ref{eq:x=z+1/z}). 
We remark  that solving 
the above equation as a quadratic
equation for $z$ yields
$$
z = \frac{x-\sqrt{x^2-4}}{2}
=\frac{x}{2}\left(1-\sqrt{1-\left(\frac{2}{x}
\right)^2}
\right)
=
\frac{x}{2}
\sum_{m=1}^\infty (-1)^{m-1} \binom{\half}{m}
\left(\frac{2}{x}
\right)^{2m},
$$
from which the closed formula
 (\ref{eq:Catalan formula}) follows.

\section{The Laplace transform 
of the generalized Catalan numbers}
\label{sect:Laplace}

Let us compute the Laplace transform 
of the generalized Catalan numbers. 
Why are we interested in the Laplace transform?
The answer  becomes clear only after 
we examine the result of  computation.

So we define the discrete Laplace transform
\begin{equation}
\label{eq:FgnC}
F_{g,n}^C(t_1,\dots, t_n)
=\sum_{(\mu_1,\dots,\mu_n)\in\bZ_+^n}
D_{g,n}(\mu_1,\dots,\mu_n)\;e^{-\la w,\mu\ra}
\end{equation}
for $(g,n)$ subject to $2g-2+n>0$,
where the Laplace  dual coordinates 
$w=(w_1,\dots,w_n)$ of
$(\mu_1,\dots,\mu_n)$
is
related to the function coordinate 
$t=(t_1,\dots,t_n)$  by
\begin{equation}
\label{eq:wxzt}
e^{w_i} = x_i = z_i+\frac{1}{z_i} =
 \frac{t_i+1}{t_i-1}
+\frac{t_i-1}{t_i+1}, \qquad i=1,2,\dots,n,
\end{equation}
and $\la w,\mu\ra = w_1\mu_1+\cdots+w_n\mu_n$.
The \textbf{Eynard-Orantin differential form}
of type $(g,n)$ is given by
\begin{equation}
\label{eq:Wgn}
\begin{aligned}
W_{g,n}^C(t_1,\dots,t_n) &= 
d_1\cdots d_n F_{g,n}^C(t_1,\dots,t_n)\\
&=(-1)^n \sum_{(\mu_1,\dots,\mu_n)\in\bZ_+^n}
C_{g,n}(\mu_1,\dots,\mu_n)\;e^{-\la w,\mu\ra}
dw_1\cdots dw_n
\end{aligned}.
\end{equation}

Due to the irregularity that a single point
is a connected cellular graph of type $(0,1)$, 
we \emph{define}
\begin{equation}
\label{eq:W01}
W_{0,1}^C(t) = -\sum_{\mu=0}^\infty
C_{0,1}(\mu) \frac{1}{x^{\mu}} \cdot \frac{dx}{x}
=-z(x)dx,
\end{equation}
including the $\mu=0$ term. Since $dF_{0,1}^C
=W_{0,1}^C$, we find
\begin{equation}
\label{eq:F01}
F_{0,1}^C(t) = -\frac{1}{2}z^2 + \log z +\const.
\end{equation}
Using the value of Kodama and Pierce
\cite{KP}  for
$D_{0,2}(\mu_1,\mu_2)$,
we calculate (see \cite{DMSS})
\begin{equation}
\label{eq:F02}
F_{0,2}^C(t_1,t_2) = -\log(1-z_1z_2),
\end{equation}
and hence 
\begin{equation}
\label{eq:W02}
W_{0,2}^C(t_1,t_2) = \frac{dt_1\cdot dt_2}
{(t_1-t_2)^2}-\frac{dx_1\cdot dx_2}
{(x_1-x_2)^2} =
\frac{dt_1\cdot dt_2}
{(t_1+t_2)^2}.
\end{equation}
The $2$-form $\frac{dx_1\cdot dx_2}
{(x_1-x_2)^2}$ is the local expression of the
symmetric 
second derivative of the logarithm of
Riemann's \textbf{prime form} on a
Riemann surface. Thus $W_{0,2}^C$ is the 
difference of this quantity between the 
Riemann surface of $x=z+\frac{1}{z}$
and the $x$-coordinate plane. 
This relation is true for all known 
examples, and hence $W_{0,2}$ is
\emph{defined} as the second log
derivative of the prime form of the spectral curve
in \cite{EO1}. It is important
to note that in our definition,
$W_{0,2}^C(t_1,t_2)$ is regular at the diagonal
$t_1=t_2$.

Note that the function $z(x)$ is absolutely
convergent for $|x|>2$. Since its inverse
function is a rational function
given by (\ref{eq:x=z+1/z}), 
the \emph{Riemann surface} of the inverse function,
i.e., the maximal domain of holomorphy of $x(z)$, 
is $\bP^1\setminus \{0,\infty\}$. At $z=\pm 1$
the function $x=z+\frac{1}{z}$ is branched, and
this is why $z(x)$ has the radius of convergence
$2$, measured from $\infty$. The coordinate
change 
$$
z=\frac{t+1}{t-1}
$$
brings the branch points to $0$ and $\infty$.

\begin{thm}[\cite{MZhou}]
\label{thm:FC recursion}
The Laplace transform $F^C_{g,n}(t_{[n]})$ 
satisfies the following
differential recursion equation
for every $(g,n)$ subject to $2g-2+n>0$.
\begin{multline}
\label{eq:FC recursion}
\frac{\partial}{\partial t_1}F^C_{g,n}(t_{[n]})
\\
=
-\frac{1}{16}
\sum_{j=2} ^n
\left[\frac{t_j}{t_1^2-t_j^2}
\left(
\frac{(t_1^2-1)^3}{t_1^2}\frac{\partial}{\partial t_1}
F^C_{g,n-1}(t_{[\hat{j}]})
-
\frac{(t_j^2-1)^3}{t_j^2}\frac{\partial}{\partial t_j}
F^C_{g,n-1}(t_{[\hat{1}]})
\right)
\right]
\\
-\frac{1}{16}
\sum_{j=2} ^n
\frac{(t_1^2-1)^2}{t_1^2}\frac{\partial}{\partial t_1}
F^C_{g,n-1}(t_{[\hat{j}]})
\\
-
\frac{1}{32}\;\frac{(t_1^2-1)^3}{t_1^2}
\left.
\left[
\frac{\partial^2}{\partial u_1\partial u_2}
F^C_{g-1,n+1}(u_1,u_2,t_2, t_3,\dots,t_n)
\right]
\right|_{u_1=u_2=t_1}
\\
-
\frac{1}{32}\;\frac{(t_1^2-1)^3}{t_1^2}
\sum_{\substack{g_1+g_2=g\\
I\sqcup J=\{2,3,\dots,n\}}}
^{\rm{stable}}
\frac{\partial}{\partial t_1}
F^C_{g_1,|I|+1}(t_1,t_I)
\frac{\partial}{\partial t_1}
F^C_{g_2,|J|+1}(t_1,t_J).
\end{multline}
Here we use the index convention 
$[n]=\{1,2,\dots,n\}$ and $[\hat{j}] = 
\{1,2,\dots,\hat{j},\dots, n\}$.
The final sum is for partitions
subject to the stability condition
$2g_1-1+|I|>0$ and $2g_2-1+|J|>0$.
\end{thm}

The proof follows from the Laplace 
transform of (\ref{eq:Catalan gn recursion}).
Since the formula for the 
generalized Catalan numbers contain 
unstable geometries $(g,n) = (0,1)$ and $(0,2)$, 
we need to substitute the values
(\ref{eq:F01}) and (\ref{eq:F02}) in the computation
to derive the recursion in the form of
(\ref{eq:FC recursion}).

Since the form of the equation (\ref{eq:FC recursion})
is identical to \cite[Theorem~5.1]{MP2012}, and 
since
the initial values $F_{1,1}^C$ and 
$F_{0,3}^C$ of \cite{MZhou} agree with 
that of
\cite[(6.1), (6,2)]{MP2012},
the same conclusion of \cite{MP2012} holds.
Therefore, 

\begin{thm}
\label{thm:geometry}
The Laplace transform $F_{g,n}^C(t_1,\dots,t_n)$
in the stable range $2g-2+n>0$
satisfies the following properties.
\begin{itemize}
\item The reciprocity:
$F_{g,n}^C(1/t_1,\dots,1/t_n) =
F_{g,n}^C(t_1,\dots,t_n) $. 

\item The polynomiality:
$F_{g,n}^C(t_1,\dots,t_n)$ is a Laurent polynomial
of degree $3(2g-2+n)$.
\item The highest degree asymptotics as the 
Virasoro condition:
The leading terms of 
$F_{g,n}^C(t_1,\dots,t_n)$ form a homogeneous 
polynomial defined by
\begin{equation}
\label{eq:FCtop}
F_{g,n}^{C\text{-top}}(t_1,\dots,t_n)
=\frac{(-1)^n}{2^{2g-2+n}}
\sum_{\substack{d_1+\cdots+d_n\\=3g-3+n}}
\la \tau_{d_1}\cdots\tau_{d_n}\ra_{g,n}
\prod_{i=1}^n \left[
(2d_i-1)!!\left(\frac{t_i}{2}\right)^{2d_i+1}
\right],
\end{equation}
where $\la \tau_{d_1}\cdots\tau_{d_n}\ra_{g,n}$
is the $\psi$-class intersection numbers of
the Deligne-Mumford moduli stack
$\Mbar_{g,n}$. The recursion
Theorem~\ref{thm:FC recursion} restricts
to the highest degree terms and produces
the DVV formulation \cite{DVV} of the
Witten-Kontsevich theorem
\cite{K1992, W1991}, which
is equivalent to the Virasoro constraint condition
for the intersection numbers on $\Mbar_{g,n}$.
\item The Poinar\'e polynomial:
The principal specialization 
$F_{g,n}^C(t,t,\dots,t)$ is  a polynomial
in 
\begin{equation}
\label{eq:s}
s = \frac{(t+1)^2}{4t}, 
\end{equation}
and coincides with 
the virtual Poincar\'e polynomial of 
$\cM_{g,n}\times \bR_+ ^n$.
\item The Euler characteristic:
In particular, we have
$$
F_{g,n}^C(1,1\dots,1) = (-1)^n \rchi(\cM_{g,n}).
$$
\end{itemize}
\end{thm}

\begin{rem}
The above theorem explains why  
 the Laplace transform of the
generalized Catalan numbers is important. 
The function
$F_{g,n}^C(t_1,\dots,t_n)$ knows a lot of 
topological information of both 
 $\cM_{g,n}$ {and} 
$\Mbar_{g,n}$. 
\end{rem}

Taking the $n$-fold differentiation of
(\ref{eq:FC recursion}), we obtain a 
residue form of the recursion. The formula given in
(\ref{eq:CEO}) is an example of the
Eynard-Orantin topological recursion.

\begin{thm}[\cite{DMSS}]
The Laplace transform of the Catalan numbers of
type $(g,n)$ defined as a symmetric differential form
$$
W_{g,n}^C(t_1,\dots,t_n) =
(-1)^n
\sum_{(\mu_1,\dots,\mu_n)\in\bZ_+ ^n}
C_{g,n}(\mu_1,\dots,\mu_n) \;e^{-\la w,\mu\ra}
dw_1\cdots dw_n
$$
satisfies the Eynard-Orantin recursion with respect
to the Lagrangian immersion
\begin{equation}
\label{eq:Catalan immersion}
\Sigma=\bC \owns z\longmapsto
(x(z),y(z))\in T^*\bC,
\qquad 
\begin{cases}
x(z) = z+\frac{1}{z}\\
y(z) = -z
\end{cases}.
\end{equation}
The recursion formula is given by
a residue transformation equation
\begin{multline}
\label{eq:CEO}
W_{g,n}^C(t_1,\dots,t_n)
=
\frac{1}{2\pi i}\int_\gam
K^C(t,t_1)
\Bigg[
\sum_{j=2}^n
\bigg(
W_{0,2}^C(t,t_j)W_{g,n-1}^C
(-t,t_2,\dots,\widehat{t_j},
\dots,t_n)
\\
+
W_{0,2}^D(-t,t_j)W_{g,n-1}^C
(t,t_2,\dots,\widehat{t_j},
\dots,t_n)
\bigg)
\\
+
W_{g-1,n+1}^C(t,{-t},t_2,\dots,t_n)
+
\sum^{\text{stable}} _
{\substack{g_1+g_2=g\\I\sqcup J=\{2,3,\dots,n\}}}
W_{g_1,|I|+1}^C(t,t_I) W_{g_2,|J|+1}^C({-t},t_J)
\Bigg].
\end{multline}
The kernel function is defined to be
\begin{equation}
\label{eq:kernel}
K^C(t,t_1) = 
\half\;\frac{\int_t ^{-t}W_{0,2}( \;\cdot\;,t_1)}
{W_{0,1}(-t)-W_{0,1}(t)}=
-\frac{1}{64} 
\left(
\frac{1}{t+t_1}+\frac{1}{t-t_1}
\right)
\frac{(t^2-1)^3}{t^2}\cdot \frac{1}{dt}\cdot dt_1,
\end{equation}
which is an algebraic operator contracting
$dt$, while multiplying $dt_1$.
The 
contour integration is taken with respect to 
$t$ on the curve defined in Figure~\ref{fig:contourC}.
\end{thm}

\begin{figure}[htb]
\centerline{\epsfig{file=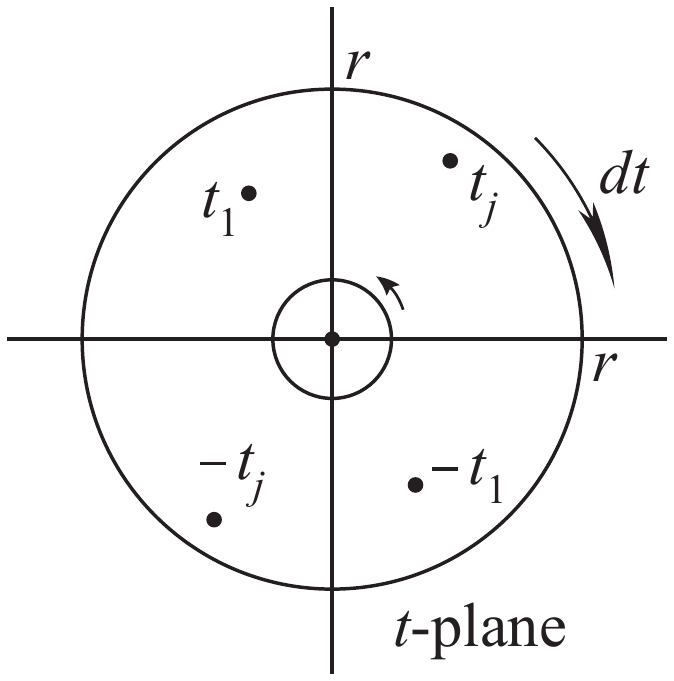, width=1.5in}}
\caption{The integration contour $\gamma$.}
\label{fig:contourC}
\end{figure}

\begin{rem}
The recursion  (\ref{eq:CEO}) is 
a universal formula compared to 
(\ref{eq:FC recursion}), because the only
input is the spectral curve $\Sigma$ 
that is realized as a 
Lagrangian immersion, which determines
$W_{0,1}$, and $W_{0,2}$ can be 
defined by taking the difference of the log
of prime forms of $\Sigma$ and $\bC$.
\end{rem}

\section{The partition function for the
generalized Catalan numbers and the
Schr\"odinger equation}

Let us now consider the 
exponential generating function 
of the Poincar\'e polynomial
$F_{g,n}^C(t,\dots,t)$. This function 
is called the \textbf{partition function}
 for the generalized Catalan numbers:
\begin{equation}
\label{eq:ZC}
Z^C(t,\hbar)=\exp\left(
\sum_{g=0}^\infty\sum_{n=1}^\infty
\frac{1}{n!}\;
\hbar ^{2g-2+n}F_{g,n}^C(t,t,\dots,t)
\right).
\end{equation}
The constant ambiguity in (\ref{eq:F01})
makes the partition function well defined up 
an overall non-zero constant factor.

\begin{thm}[\cite{MSul}]
\label{thm:Sch}
The 
partition function satisfies the following
Schr\"odinger equation
\begin{equation}
\label{eq:Sch}
\left(
\hbar^2 \frac{d^2}{dx^2}+\hbar x\frac{d}{dx}+1
\right)
Z^C(t,\hbar) = 0,
\end{equation}
where $t$ is considered as a function in $x$
by
$$
t= t(x) = \frac{z(x)+1}{z(x)-1}
$$
and  {\rm{(\ref{eq:Catalan z})}}. Moreover,
the partition function has a matrix
integral expression
\begin{equation}
\label{eq:matrix}
Z^C(z,\hbar) = \int_{\cH_{N\times N}}
\det(1-\sqrt{s}X)^N e^{-\frac{N}{2}\trace (X^2)}
dX
\end{equation}
with the identification {\rm{(\ref{eq:s})}} and 
$\hbar = 1/N$. Here $dX$ is the normalized
Lebesgue measure on the space of $N\times N$
Hermitian matrices $\cH_{N\times N}$.
It is a well-known fact that this matrix integral
 is  the principal 
specialization of a KP $\tau$-function 
\cite{M1994}.
\end{thm}

The currently emerging picture
\cite{BE2,DFM,GS} is the following. 
If we start with the A-polynomial of a
knot $K$ and consider the Lagrangian 
immersion it defines, like
the one in (\ref{eq:Catalan immersion}),
then the partition
function $Z$ of the Eynard-Orantin recursion,
defined in a much similar way as in 
(\ref{eq:ZC}) but with a theta function correction
factor of \cite{BE2}, \emph{is} the 
colored Jones polynomial of $K$, and
the corresponding
Schr\"odinger equation like (\ref{eq:Sch}) is
equivalent to the AJ-conjecture of 
\cite{Gar}.

Our example comes from an elementary
enumeration problem, yet as 
Theorem~\ref{thm:geometry} suggests, the
geometric information contained in this 
example is quite non-trivial.

\begin{ack}
The paper is based on the author's talk at the 
\emph{XXXI Workshop on the 
Geometric Methods in 
Physics} held in Bia\l owie\.za, Poland, 
in June 2012. He thanks the organizers
of the workshop for their hospitality and 
exceptional organization of the successful
workshop. 
The author also thanks
Ga\"etan Borot,
Vincent Bouchard,
Bertrand Eynard, 
Marcos Mari\~no,
Paul Norbury,
Yongbin Ruan,
Sergey Shadrin,
Piotr Su\l kowski, and
Don Zagier
for their tireless and patient explanations
of their work to the author, and for 
stimulating discussions.
The author's research 
was
 supported by NSF grants
DMS-1104734
and DMS-1104751.
\end{ack}


\providecommand{\bysame}{\leavevmode\hbox to3em{\hrulefill}\thinspace}

\bibliographystyle{amsplain}

\end{document}